\documentclass[11pt]{amsproc}

\usepackage{amsmath, amsthm, amssymb, mathtools, mathrsfs, stmaryrd}
\usepackage{ascmac}
\usepackage{comment}
\usepackage{bm}
\allowdisplaybreaks

\usepackage{graphicx}
\usepackage[top=30mm, bottom=35mm, left=29mm, right=29mm]{geometry}

\usepackage{hyperref}

\usepackage{here}
\usepackage{time}
\usepackage[abbrev]{amsrefs}

\usepackage{xcolor}
\usepackage[capitalize,nameinlink,noabbrev,nosort]{cleveref}
\hypersetup{
	colorlinks=true,       
	linkcolor=brown,          
	citecolor=brown,        
	filecolor=brown,      
	urlcolor=brown,           
}

\makeatletter
\@namedef{subjclassname@2020}{\textup{2020} Mathematics Subject Classification}
\makeatother

\newtheorem*{theorem*}{\hspace{-6.3mm}\textbf{Theorem}}  

\newtheorem{theoremcounter}{Theorem Counter}[section]

\theoremstyle{remark}
\newtheorem*{remark*}{Remark}

\theoremstyle{definition}
\newtheorem{definition}[theoremcounter]{Definition}
\newtheorem*{example*}{Example}

\theoremstyle{plain}
\newtheorem{lemma}[theoremcounter]{Lemma}
\newtheorem{proposition}[theoremcounter]{Proposition}
\newtheorem{corollary}[theoremcounter]{Corollary}

\newtheorem{theorem}[theoremcounter]{Theorem}

\numberwithin{equation}{section}

\newcommand{\Z}{\mathbb{Z}}

\newcommand{\R}{\mathbb{R}}
\newcommand{\C}{\mathbb{C}}

\newcommand{\G}{\Gamma}
\newcommand{\g}{\gamma}
\newcommand{\tr}{\mathrm{tr}}

\newcommand{\bbH}{\mathbb{H}}

\DeclareMathOperator{\ImNew}{Im}
\renewcommand{\Im}{\ImNew}

\DeclareMathOperator{\SL}{SL}
\DeclareMathOperator{\PSL}{PSL}

\DeclareMathOperator{\sgn}{sgn}

\newcommand{\pmat}[1]{\begin{pmatrix}#1\end{pmatrix}}

\newcommand{\smat}[1]{\bigl(\begin{smallmatrix}#1\end{smallmatrix}\bigr)}

\begin{document}

\title[]{Note on an explicit formula of Rademacher symbols\\ for triangle groups} 

\author{Toshiki Matsusaka}
\address{Faculty of Mathematics, Kyushu University, Motooka 744, Nishi-ku, Fukuoka 819-0395, Japan}
\email{matsusaka@math.kyushu-u.ac.jp}

\author{Gyucheol Shin}
\address{Department of Mathematics, Sungkyunkwan University, Suwon 16419, Korea}
\email{sgc7982@gmail.com}





\maketitle

\begin{abstract}
	The purpose of this note is to present an explicit formula of the Rademacher symbols for triangle groups. This result generalizes Ghys' third proof of the identity relating to the linking numbers of modular knots.
\end{abstract}


\section{Introduction}

Over a century ago, in 1892, Richard Dedekind~\cite{Dedekind1892} explored functional equations of $\log \eta(z)$, defined as the generating function of the $(-1)$-th divisor sum
\[
	\log \eta(z) \coloneqq \frac{\pi iz}{12} - \sum_{n=1}^\infty \sum_{d \mid n} d^{-1} e^{2\pi inz} \quad (z \in \bbH),
\]
where $\bbH \coloneqq \{z \in \C : \Im(z) > 0\}$ is the upper-half plane. Specifically, for any $\gamma = \smat{a & b \\ c & d} \in \SL_2(\Z)$ with $c \neq 0$, Dedekind derived the following formula.
\[
	\log \eta \left(\frac{az+b}{cz+d}\right) = \log \eta(z) + \frac{1}{2} \log \left(\frac{cz+d}{i \sgn c} \right) + \frac{\pi i}{12} \Phi(\gamma),
\]
where we take a branch $\Im \log(z) \in [-\pi, \pi)$, and $\Phi : \SL_2(\Z) \to \Z$ is defined by
\[
	\Phi(\gamma) = \frac{a+d}{c} - 12 \sgn c \sum_{k=1}^{|c|-1} \left(\!\!\left( \frac{k}{c} \right)\!\!\right)  \left(\!\!\left( \frac{ka}{c} \right)\!\!\right)
\]
with the notation $(\!(x)\!) = x - \lfloor x \rfloor - 1/2$ if $x \not\in \Z$ and $(\!(x)\!) = 0$ if $x \in \Z$. For $c = 0$, we set $\Phi(\gamma) = b/d$. As Hans Rademacher noted~\cite[p.1]{Rademacher1972}, at first glance, Dedekind's $\Phi(\gamma)$ may seem like a specialized subject. However, it has been extensively studied over time as a highly ubiquitous function within a broad context, as detailed in Rademacher's monograph~\cite{Rademacher 1972} and Atiyah's omnibus theorem~\cite[p.371]{Atiyah1987}. 

The \emph{Rademacher symbol} $\Psi: \SL_2(\Z) \to \Z$, the focus of this note, is a slight modification of Dedekind's function. It is defined by $\Psi(\gamma) \coloneqq \Phi(\gamma) - 3\sgn(c(a+d))$, and as discovered by Rademacher~\cite{Rademacher1956} in 1956, this modification ensures that $\Psi(\gamma)$ is a conjugacy class invariant on $\PSL_2(\Z) \coloneqq \SL_2(\Z)/\{\pm I\}$.

\begin{theorem*}[{\cite[Satz 7]{Rademacher1956}}]
	For any $\gamma, g \in \SL_2(\Z)$, we have $\Psi(\gamma) = \Psi(-\gamma) = \Psi(g^{-1} \gamma g)$.
\end{theorem*}
Additionally, he provided its simple explicit formula in terms of the generators $S = \smat{0 & -1 \\ 1 & 0}$ and $U = \smat{1 & -1 \\ 1 & 0}$ of $\SL_2(\Z)$. 
\begin{theorem*}[{\cite[Satz 8]{Rademacher1956}}]
	For $\gamma \in \SL_2(\Z)$ conjugate to $\pm S U^{\varepsilon_1} \cdots S U^{\varepsilon_r}$ with $\varepsilon_j \in \{-1, +1\}$, we have
	\[
		\Psi(\gamma) = \sum_{j=1}^r \varepsilon_j,
	\]
\end{theorem*}

This formula has found surprising applications, notably in Ghys' study of modular knots~\cite{Ghys2007}. Ghys discovered that the linking number, an invariant of links, coincides with the Rademacher symbol. To be more precise, \emph{modular knots} are defined as closed orbits of the geodesic flow $\varphi^t : M \mapsto M \smat{e^t & 0 \\ 0 & e^{-t}}$ ($t \in \R$) defined on the homogeneous space $\SL_2(\Z) \backslash \SL_2(\R)$. It is known that modular knots can be parameterized by elements $\gamma \in \SL_2(\Z)$ such that $\mathrm{tr} (\gamma) > 2$. As Milnor~\cite[p.84]{Milnor1971} showed, with credit given to Quillen, this homogeneous space is homeomorphic to the complement of the trefoil knot $K_{2,3}$ in the three-dimensional sphere $S^3$, i.e., $\SL_2(\Z) \backslash \SL_2(\R) \cong S^3 - K_{2,3}$. That is to say, for each $\gamma \in \SL_2(\Z)$ with $\mathrm{tr}(\gamma) > 2$, a two-component link in $S^3$ is formed by the modular knot $C_\gamma$ and the trefoil knot $K_{2,3}$. Then, Ghys proved the following identity.

\begin{theorem*}[{\cite[Section 3]{Ghys2007}}]
	For every $\gamma \in \SL_2(\Z)$ with $\mathrm{tr}(\gamma) > 2$, the linking number $\mathrm{lk}(C_\gamma, K_{2,3})$ between $C_\gamma$ and $K_{2,3}$ is equal to the Rademacher symbol $\Psi(\gamma)$.
\end{theorem*}

Ghys offered three different proofs of this result, the third of which is our main focus. By employing \emph{template theory}, Ghys traces the modular knot orbits and compares them to the combinatorial definition of the linking numbers. He showed that for $\gamma$ conjugate to $\pm SU^{\varepsilon_1} \cdots SU^{\varepsilon_r}$ with $\varepsilon_j \in \{-1, +1\}$, the linking number is given by
\[
	\mathrm{lk}(C_\gamma, K_{2,3}) = \sum_{j=1}^r \varepsilon_j.
\]
This result coincides precisely with Rademacher's explicit formula. This is his third proof.

Building on this work, the first author and Ueki~\cite{MatsusakaUeki2023} extended Ghys' result from $\SL_2(\Z)$ to an infinite family of triangle groups $\Gamma_{p,q}$, where $(p,q)$ is a coprime pair of integers with $2 \le p < q$. The \emph{triangle group} is a discrete subgroup of $\SL_2(\R)$ generated by two torsion elements $S_p, U_q$, satisfying the relations $S_p^p = U_q^q = -I$. In particular, when $(p,q) = (2,3)$, we have $\Gamma_{2,3} = \SL_2(\Z)$, thus generalizing the classical case. Similar to Milnor's result, it is known that the quotient space $\Gamma_{p,q} \backslash \SL_2(\R)$ is homeomorphic to the complement of the $(p,q)$-torus knot\footnote{Precisely, for a certain discrete subgroup $G$ of the universal covering group $\widetilde{SL}_2(\R)$ of $\SL_2(\R)$, the complement of the $(p,q)$-torus knot $K_{p,q}$ in $S^3$ is given as $G \backslash \widetilde{\SL}_2(\R) \cong S^3 - K_{p,q}$, and the projection of $K_{p,q}$ to the lens space $L(pq-p-q,p-1)$ is denoted by $\overline{K}_{p,q}$, (see~\cite{MatsusakaUeki2023}).} $\overline{K}_{p,q}$ in the lens space $L(pq-p-q, p-1)$. The concept of modular knots can similarly be extended as closed orbits of the same geodesic flow. Based on ideas from Goldstein~\cite{Goldstein1973} and more recent work by Burrin~\cite{Burrin2018, Burrin2022}, the first author and Ueki introduced a generalization of the Rademacher symbol $\Psi_{p,q}: \Gamma_{p,q} \to \Z$ for triangle groups, (a precise definition will be given in \cref{sec:Rademacher}). By extending Ghys' first proof through the theory of modular forms, they established a generalized identity between the linking number and the Rademacher symbol.

In Section 5 of~\cite{MatsusakaUeki2023}, a generalization of Rademacher's explicit formula was proposed as a future problem during their discussion of Ghys' third proof. The main theorem of this note addresses this problem and completes the generalization of Ghys' third proof.

\begin{theorem}\label{thm:main}
	For $\gamma \in \Gamma_{p,q}$ conjugate to $\pm S_p^{n_1} U_q^{m_1} \cdots S_p^{n_r} U_q^{m_r}$ with $0 < n_j < p$ and $0 < m_j < q$, we have
	\[
		\Psi_{p,q}(\gamma) = \sum_{j=1}^r (pq - qn_j - pm_j).
	\]
\end{theorem}

The calculation of the linking number using template theory has already been extended by Dehornoy~\cite[Proposition 5.7]{Dehornoy2015}, who showed that the linking number of a modular knots $C_\gamma$ in the knot complement $L(pq-p-q, p-1) - \overline{K}_{p,q} \cong \Gamma_{p,q} \backslash \SL_2(\R)$ is given by
\[
	\mathrm{lk}(C_\gamma, \overline{K}_{p,q}) = \frac{pq}{pq-p-q} \sum_{j=1}^r \left(\frac{p/2-n_j}{p} + \frac{q/2-m_j}{q}\right),
\]
(with the sign convention adjusted to match that in~\cite{MatsusakaUeki2023}). By comparison, this formula yields the result of Matsusaka--Ueki~\cite[Theorem 4.10]{MatsusakaUeki2023}.

\begin{corollary}
	For every $\gamma \in \Gamma_{p,q}$ with $\mathrm{tr}(\gamma) > 2$, we have
	\[
		\mathrm{lk}(C_\gamma, \overline{K}_{p,q}) = \frac{1}{pq-p-q} \Psi_{p,q}(\gamma).
	\]
\end{corollary}

This note is organized as follows. In \cref{sec:Rademacher}, we define the Rademacher symbol $\Psi_{p,q}(\gamma)$ for triangle groups and outline some properties needed to prove the main theorem. We will not explain further on modular knots here, but we recommend referring the original article~\cite{Ghys2007}, Simon's Ph.D.~thesis~\cite{Simon2022}, and the previous work~\cite{MatsusakaUeki2023} for precise definitions. In \cref{sec:Proof}, we prove our main theorem.

\section*{Acknowledgments}
The first author was supported by JSPS KAKENHI (JP21K18141 and JP24K16901) and by the MEXT Initiative through Kyushu University's Diversity and Super Global Training Program for Female and Young Faculty (SENTAN-Q). The second author was supported by the National Research Foundation of Korea (NRF) grant funded by the Korea government (MSIT) (RS-2024-00348504).
\section{The Rademacher symbols for triangle groups} \label{sec:Rademacher}

In~\cite[Theorem B]{MatsusakaUeki2023}, the authors provided several different definitions of the Rademacher symbol. Here, we aim to introduce it using Asai's cocycle $W$. Let $(p,q)$ be as before. We define the triangle group $\Gamma_{p,q} = \Gamma(p,q,\infty)$ as a subgroup of $\SL_2(\R)$ generated by the two matrices
\[
	S_p \coloneqq \pmat{0 & -1 \\ 1 & 2\cos (\pi/p)}, \quad U_q \coloneqq \pmat{2 \cos (\pi/q) & -1 \\ 1 & 0}
\]
satisfying the relations $S_p^p = U_q^q = -I$. The following formulas are known.
\begin{lemma}[{\cite[Lemma 2.2]{MatsusakaUeki2023}}]\label{recur}
	For every integer $n \in \Z$, we define the Chebyshev polynomial of the second kind $C_n(x) \in \Z[x]$ by setting $C_0(x) = 0, C_1(x) = 1$, and using the recursion $C_{n+1}(x) = 2x C_n(x) - C_{n-1}(x)$. Then we have that $C_n(\cos t) = \sin nt/\sin t$ and the generators $S_p$ and $U_q$ satisfy
	\[
		S_p^n = \pmat{-C_{n-1}(s) & -C_n(s) \\ C_n(s) & C_{n+1}(s)}, \quad U_q^n = \pmat{C_{n+1}(u) & -C_n(u) \\ C_n(u) & -C_{n-1}(u)},
	\]
	where we put $s = \cos (\pi/p)$ and $u = \cos (\pi/q)$.
\end{lemma}

In 1970, Asai~\cite{Asai1970} examined the following function to characterize the Rademacher symbol more naturally from a group-theoretic perspective.

\begin{definition}
	Let $j: \SL_2(\R) \times \bbH \to \C$ be the automorphic factor defined by $j(\smat{a & b \\ c & d}, z) = cz+d$. Then, we define $W: \SL_2(\R) \times \SL_2(\R) \to \Z$ by
	\[
		W(g_1, g_2) \coloneqq \frac{1}{2\pi i} \Big(\log j(g_1, g_2 z) + \log j(g_2, z) - \log j(g_1 g_2, z) \Big),
	\]
	where we take a branch $\Im \log j(g, z) \in [-\pi, \pi)$.
\end{definition}

To elaborate on the definition, it is first derived from the well-known relation $j(g_1 g_2, z) = j(g_1, g_2 z) j(g_2, z)$. Consequently, since $e^{2\pi iW(g_1, g_2)} = 1$, it follows that $W(g_1, g_2) \in \Z$. Moreover, because the defining equation for $W$ is continuous in $z \in \bbH$, it does not depend on the choice of $z$. 

By direct calculation, we can verify that
\[
	W(g_1 g_2, g_3) + W(g_1, g_2) = W(g_1, g_2 g_3) + W(g_2, g_3)
\]
for any $g_1, g_2, g_3 \in \SL_2(\R)$. This shows that $W$ is a $2$-cocycle. Additionally, based on the convention on angles, we find that $-3\pi < 2\pi W(g_1, g_2) < 3\pi$, which implies that $W(g_1, g_2) \in \{-1, 0, 1\}$.

We can also describe the value of $W$ more explicitly. According to Asai, adopting the branch of the logarithm as $[-\pi, \pi)$ is a key point, as this choice allows us to express the value of $W$ in a ``wonderfully simple" form. To describe this, we introduce the sign function $\sgn: \SL_2(\R)\rightarrow\{-1,1\}$ defined by
\[
	\sgn (g) \coloneqq \begin{cases}
		\sgn c &\text{if } c\ne0,\\
		\sgn d &\text{if } c=0
	\end{cases}
\]
for $g = \smat{a & b \\ c & d} \in \SL_2(\R)$. Then, Asai showed the following formula.

\begin{proposition}[{Asai~\cite[Theorem 2]{Asai1970}}]
	For any $g_1, g_2 \in \SL_2(\R)$, we have
	\begin{align}\label{eq:W-explicit}
		W(g_1,g_2) &= \begin{cases}
			1 & \text{if } \sgn(g_1)=\sgn(g_2)=1,\sgn(g_1 g_2)=-1,\\
			-1 & \text{if } \sgn(g_1)=\sgn(g_2)=-1, \sgn(g_1g_2)=1,\\
			0 & \text{if otherwise}
		\end{cases}\\ \label{Wsgn}
		&= \frac{1}{4}\Big(\sgn(g_{1})+\sgn(g_{2})-\sgn(g_{1}g_{2})-\sgn(g_{1})\sgn(g_{2})\sgn(g_{1}g_{2})\Big).
	\end{align}
\end{proposition}

As discovered by Asai~\cite[Theorem 3]{Asai1970} and generalized to triangle groups in~\cite[Theorem 3.9]{MatsusakaUeki2023}, there exists a unique function $\psi_{p,q} : \Gamma_{p,q} \to \Z$ that satisfies the relation
\begin{align}\label{definitionofpsi_{p,q}}
	\psi_{p,q}(\g_{1}\g_{2}) = \psi_{p,q}(\g_{1})+\psi_{p,q}(\g_{2})+2pqW(\g_{1},\g_{2})
\end{align}
for any $\gamma_1, \gamma_2 \in \Gamma_{p,q}$. In particular, for the generators, it follows that $\psi_{p,q}(S_p) = -q$ and $\psi_{p,q}(U_q) = -p$, and we have $\psi_{2,3}(\gamma) = \Psi(\gamma)$ for $\gamma \in \SL_2(\Z)$ with $\mathrm{tr}(\gamma) > 0$. The modification of this function to make it a conjugacy class invariant, based on Rademacher, gives rise to the \emph{Rademacher symbol for triangle groups}, as indicated in the title of this note.

\begin{definition}
	We define the Rademacher symbol $\Psi_{p,q}:\G_{p,q}\rightarrow\Z$ by
	\[
		\Psi_{p,q}(\g) \coloneqq \psi_{p,q}(\g)+\frac{pq}{2}\sgn(\g)(1-\sgn\tr(\g)).
	\]
\end{definition}

We know that $\Psi_{2,3}(\gamma) = \Psi(\gamma)$ for $\Gamma_{2,3} = \SL_2(\Z)$, and as shown in~\cite[Proposition 3.16]{MatsusakaUeki2023}, this is a conjugacy class invariant on $\Gamma_{p,q}/\{\pm I\}$, that is, for any $\gamma, g \in \Gamma_{p,q}$, we have
\begin{align}\label{eq:class-inv}
	\Psi_{p,q}(\gamma) = \Psi_{p,q}(-\gamma) = \Psi_{p,q}(g^{-1} \gamma g).
\end{align}

Next, we will prepare several lemmas regarding the Asai cocycle $W$ that are necessary for proving \cref{thm:main}.

\begin{lemma}\label{positivity}
	Let $n$ be a positive integer such that $0<n\leq p$. Then we have
	\[
		\sgn(S_{p}^{n}) = \begin{cases}
			-1 & \text{if } n=p,\\
			1 & \text{if } 0<n<p,
		\end{cases}
		\quad
		W(S_{p},S_{p}^{n})=\begin{cases}
			1 & \text{if } n=p-1, \\
			0 & \text{if } 0<n<p-1.
		\end{cases}
	\]
	Similar results hold for $U_{q}$, that is,
	\begin{equation*}
		\sgn(U_{q}^{m})=\begin{cases}
			-1 & \text{if } m=q,\\
		1 & \text{if } 0<m<q,
		\end{cases}
		\quad
		W(U_{q},U_{q}^{m})=\begin{cases}
			1 & \text{if } m=q-1,\\
			0 & \text{if } 0<m<q-1.
		\end{cases}
	\]
\end{lemma}

\begin{proof}
For simplicity we only prove the case of $S_{p}^{n}$. Clearly our claim is true when $n=p$. By \cref{recur}, $(2,1)$-entry of $S_{p}^{n}$ is equal to $C_n(s)$ where $s= \cos (\pi/p)$. Since $C_n(x) \in \Z[x]$ is a polynomial of degree $n-1$ satisfying $C_n(\cos t) = \sin nt/ \sin t$, we can verify that the zeros of $C_n(x)$ are located at $x = \cos \frac{\pi k}{n}$ for $k = 1, 2, \dots, n-1$. Therefore, the value $s = \cos(\pi/p)$ is greater than all zeros of $C_n(x)$ if $0 < n < p$. Since $C_n(x) = (2x)^{n-1} +$ (lower degree terms), for $n > 1$, we see that $C_n(x) \to \infty$ as $x \to \infty$. Thus, by the intermediate value theorem, we obtain $C_n(s) > 0$ and conclude $\sgn(S_p^n) = 1$ for $0 < n < p$. Finally, it follows that
\[
	W(S_{p},S_{p}^{n})= \begin{cases}
		1 & \text{if } \sgn(S_{p}^{n})=1 \text{ and } \sgn(S_{p}^{n+1})=-1,\\
		0 & \text{if otherwise}
	\end{cases}
\]
since $\sgn(S_{p})=1$. Hence, we obtain the desired result.
\end{proof}

As an application of the proof of \cref{positivity}, we obtain the following.

\begin{lemma}\label{signchange}
	If 
	\[
		M=\begin{pmatrix} a & b \\ c & d \end{pmatrix}=S_{p}^{n_{1}}U_{q}^{m_{1}} \cdots S_{p}^{n_{r}}U_{q}^{m_{r}},
	\]
	where $0<n_{j}<p$ and $0<m_{j}<q$ for all $1\leq j \leq r$, then we have
	\[
	\begin{cases}
		a, d <0 \text{ and } b,c \geq0 & \text{if $r$ is odd},\\
		a, d >0 \text{ and } b,c \leq0 & \text{if $r$ is even}.
	\end{cases}
	\]
\end{lemma}

\begin{proof}
We use induction on $r$. When $r=1$, it immediately follows from the facts that
\begin{align*}
	S_{p}^{n}U_{q}^{m} &=\begin{pmatrix}-C_{n-1}(s) C_{m+1}(u) - C_{n}(s) C_{m}(u) & C_{n-1}(s) C_{m}(u) + C_{n}(s) C_{m-1}(u) \\ C_{n}(s) C_{m+1}(u) + C_{n+1}(s) C_{m}(u) & -C_{n}(s) C_{m}(u) - C_{n+1}(s) C_{m-1}(u) \end{pmatrix},
\end{align*}
and
\[
	C_{n}(s) \begin{cases}
		>0  & \text{if } 0<n<p,\\
		=0 & \text{if } n=0, p,
	\end{cases}
	\quad
	C_{m}(u) \begin{cases}
		>0  & \text{if } 0<m<q,\\
		=0 & \text{if } m=0, q.
	\end{cases}
\]
Next, we assume that the claim holds for $M$ with length $r$. For convenience, we let $S_{p}^{n}U_{q}^{m}=\begin{psmallmatrix} a' & b' \\ c' & d' \end{psmallmatrix}$. If $r$ is odd, then we see that
\[
	MS_{p}^{n}U_{q}^{m}=\begin{pmatrix} a a'+ b c'&a b'+ b d'\\ c a'+ d c'& c b'+d d' \end{pmatrix}=\begin{pmatrix}>0&\leq0\\\leq0&>0\end{pmatrix}.
\]
The claim holds in exactly the same way when $r$ is even.
\end{proof}

\section{Proof of \cref{thm:main}} \label{sec:Proof}

By the definition of $\psi_{p,q}(\gamma)$ in \eqref{definitionofpsi_{p,q}} and \cref{positivity}, we see that
\begin{align}\label{Spn}
	\psi_{p,q}(S_{p}^{n})=n\psi_{p,q}(S_{p})+2pq\sum\limits_{j=1}^{n-1}W(S_{p},S_{p}^{j})=-nq.
\end{align}
and
\begin{align}\label{Uqm}
	\psi_{p,q}(U_{q}^{m})=m\psi_{p,q}(U_{q})+2pq\sum\limits_{j=1}^{m-1}W(U_{q},U_{q}^{j})=-mp.
\end{align}
We proceed by induction on $r$. From \eqref{Wsgn}, \cref{positivity}, and the above computations, we deduce that
\begin{align*}
\Psi_{p,q}(S_{p}^{n}U_{q}^{m})&=-nq-mp+\frac{pq}{2}\Big(2-\sgn(S_{p}^{n}U_{q}^{m})-\sgn(S_{p}^{n}U_{q}^{m})\sgn\tr(S_{p}^{n}U_{q}^{m})\Big).
\end{align*}
By \cref{signchange}, we know $\tr(S_{p}^{n}U_{q}^{m})<0$, and hence we obtain
\[
	\Psi_{p,q}(S_{p}^{n}U_{q}^{m})=pq-nq-mp,
\]
which proves the theorem when $r=1$.

Next, we let $M=\begin{psmallmatrix} a & b \\ c & d \end{psmallmatrix}\in\G_{p,q}$ be of length $r$ that satisfies the theorem, and set $S_{p}^{n}U_{q}^{m}=\begin{psmallmatrix} a' & b' \\ c' & d' \end{psmallmatrix}$. If necessary, we can replace $M$ by $-M$ to ensure $a, d < 0$ and $b, c \ge 0$ by \cref{signchange}. Using the definitions, we see that
\begin{align*}
\Psi_{p,q}(MS_{p}^{n}U_{q}^{m}) &=\psi_{p,q}(MS_{p}^{n}U_{q}^{m})+\frac{pq}{2}\sgn(MS_{p}^{n}U_{q}^{m})(1-\sgn\tr (MS_{p}^{n}U_{q}^{m}))\\
	&=\psi_{p,q}(M)+\psi_{p,q}(S_{p}^{n})+\psi_{p,q}(U_{q}^{m})+2pq\Big(W(M,S_{p}^{n}U_{q}^{m})+W(S_{p}^{n},U_{q}^{m})\Big)\\
	&\quad +\frac{pq}{2}\sgn(MS_{p}^{n}U_{q}^{m})(1-\sgn\tr (MS_{p}^{n}U_{q}^{m}))\\
	&=\Psi_{p,q}(M)+\psi_{p,q}(S_{p}^{n})+\psi_{p,q}(U_{q}^{m})+2pq\Big(W(M,S_{p}^{n}U_{q}^{m})+W(S_{p}^{n},U_{q}^{m})\Big)\\
	&\quad -\frac{pq}{2}\sgn(M)(1-\sgn\tr(M))+\frac{pq}{2}\sgn(MS_{p}^{n}U_{q}^{m})(1-\sgn\tr (MS_{p}^{n}U_{q}^{m})).
\end{align*}
Hence, it suffices to show that
\begin{align*}
	pq-nq-mp&=\psi_{p,q}(S_{p}^{n})+\psi_{p,q}(U_{q}^{m})+2pq\Big(W(M,S_{p}^{n}U_{q}^{m})+W(S_{p}^{n},U_{q}^{m})\Big)\\
	&-\frac{pq}{2}\sgn(M)(1-\sgn\tr(M))+\frac{pq}{2}\sgn(MS_{p}^{n}U_{q}^{m})(1-\sgn\tr (MS_{p}^{n}U_{q}^{m})).
\end{align*}
Equivalently, by \eqref{Spn} and \eqref{Uqm},
\begin{align}\label{eq1}
\begin{split}
	1&=2\Big(W(M,S_{p}^{n}U_{q}^{m})+W(S_{p}^{n},U_{q}^{m})\Big)\\
	&\quad -\frac{\sgn(M)}{2}(1-\sgn\tr(M))+\frac{\sgn(MS_{p}^{n}U_{q}^{m})}{2}(1-\sgn\tr(MS_{p}^{n}U_{q}^{m})).
\end{split}
\end{align}

By \eqref{Wsgn} and \cref{positivity}, we have
\begin{align}\label{rhsof3.3}
\begin{split}
	\text{RHS of \eqref{eq1}} &=\frac{1}{2}\Big(2-\sgn(S_{p}^{n}U_{q}^{m})\big(1+\sgn(M)\sgn(MS_{p}^{n}U_{q}^{m})\big)\\
	&\quad -\sgn(M) -\sgn(MS_{p}^{n}U_{q}^{m}) \Big),
\end{split}
\end{align}
where we use the fact that $M$ was chosen so that $\tr(M) < 0$ and $\tr(MS_p^n U_q^m) > 0$.
\begin{itemize}
	\item If $c>0$ and $c' \geq0$, then $\sgn(M)=1$ and $\sgn(MS_{p}^{n}U_{q}^{m})=-1$. Thus the equation \eqref{rhsof3.3} is equal to $1$.
	\item If $c=0$ and $c'>0$, then $\sgn(M)=-1$, $\sgn(S_{p}^{n}U_{q}^{m})=1$, and $\sgn(MS_{p}^{n}U_{q}^{m})=-1$. Thus the equation \eqref{rhsof3.3} is equal to $1$.
	\item If $c=c'=0$, then $\sgn(M)=-1$ and $\sgn(MS_{p}^{n}U_{q}^{m})=1$. Thus the equation \eqref{rhsof3.3} is equal to $1$.
\end{itemize}
Thus, since \eqref{eq1} holds in all cases, the proof is complete.

\begin{remark*}
	Finally, let us confirm that \cref{thm:main} includes Rademacher's result in the case of $(p,q) = (2,3)$. Noting that $U_3^2 = -U_3^{-1}$, for $0 < n < 2$ and $0 < m < 3$, we have
	\[
		pq - qn - pm = 3 - 2m = \begin{cases}
			1 &\text{if } m=1,\\
			-1 &\text{if } m=2,
		\end{cases}
	\]
	which indeed reproduces Rademacher's original explicit formula.
\end{remark*}

\bibliographystyle{amsalpha}
\bibliography{References}

\end{document}